\documentclass{amsart}
\usepackage[T1]{fontenc}
\usepackage{verbatim}
\usepackage{multirow}
\usepackage{geometry}
\usepackage{amssymb}
\usepackage{amsmath}
\usepackage{graphicx}
\usepackage{amsthm}
\usepackage{caption}
\usepackage{color}
\usepackage{enumerate}
\usepackage{graphicx}
\usepackage{float}
\usepackage{subfigure}

\newcommand{\eChar}{\begin{enumerate}[(i)]}
\newcommand{\eCharR}{\begin{enumerate}[(a)]}
\newcommand{\eBr}{\begin{enumerate}[(1)]}

\newcommand{\Abstract}

\title
{
Bounding the diameter and eigenvalues of amply regular graphs via Lin--Lu--Yau curvature
}

\author{Xueping Huang}
\address{
School of Mathematics and Statistics\\
Nanjing University of Information Science \& Technology\\
219 Ningliu Road\\
Nanjing 210044\\
Jiangsu Province\\
China}
\email{hxp@nuist.edu.cn}

\author{Shiping Liu}
\address{
School of Mathematical Sciences\\
University of Science and Technology of China\\
96 Jinzhai Road\\
Hefei 230026\\
Anhui Province\\
China}
\email{spliu@ustc.edu.cn}

\author{Qing Xia}
\email{xiaqing420@163.com}

\date{\today}
\theoremstyle{plain}
\newtheorem{lemma}{Lemma}[section]
\newtheorem{theorem}[lemma]{Theorem}
\newtheorem{proposition}[lemma]{Proposition}
\newtheorem{corollary}[lemma]{Corollary}
\theoremstyle{definition}

\newtheorem{definition}[lemma]{Definition}
\newtheorem{remark}[lemma]{Remark}

\numberwithin{equation}{section}

\begin{document}

\pagestyle{plain}

\begin{abstract}
An amply regular graph is a regular graph such that any two adjacent vertices have $\alpha$ common neighbors and any two vertices with distance $2$ have $\beta$ common neighbors.
We prove
a sharp lower bound estimate for the Lin--Lu--Yau curvature of any amply regular graph with girth $3$ and $\beta>\alpha$. The proof involves new ideas relating discrete Ricci curvature with local matching properties: This includes a novel construction of a regular bipartite graph from the local structure and related distance estimates. As a consequence, we obtain sharp diameter and eigenvalue bounds for amply regular graphs.
\end{abstract}
\maketitle

\section{Introduction and statement of results}

Bounding the diameter \cite{Terwilliger82,Terwilliger83,Ivanov83,Terwilliger85,Pyber99,NP22,NP22+} and eigenvalues \cite{Terwilliger86,BKP15,Button6,Kivva21} of a distance-regular graph in terms of its intersection numbers is a very important problem, see also \cite[Chapters 4 and 5]{BCN89} and the dynamic survey \cite{DKT16}. In particular, upper bounds for the diameter including some linear or nonlinear combination of the first two elements $\{d, b_1; 1, c_2\}$ of the intersection array have been asked, e.g., in \cite{NP22}, where $d$ stands for the vertex degree. Note that $c_2$ is the number of common neighbors of two vertices with distance $2$ and $d-b_1-1$ gives the number of common neighbors of two vertices with distance $1$. In this paper we consider this problem for a more general class of graphs called\emph{ amply regular graphs} \cite[Section 1.1]{BCN89}.
\begin{definition} [Amply regular graph \cite{BCN89}]\label{definition:amply}
We call a $d$-regular graph on $n$ vertices an amply regular graph with parameters $(n,d,\alpha,\beta)$ if it holds that:
\begin{itemize}
  \item [(i)] Any two adjacent vertices have $\alpha$ common neighbors;
  \item [(ii)] Any two vertices with distance $2$ have $\beta$ common neighbors.
\end{itemize}
\end{definition}
We restrict ourselves to connected amply regular graphs in this paper.
An amply regular graph with $\beta>1$ is called a $(2,\beta,\alpha,d)$-graph by Terwilliger \cite{Terwilliger83}.
Notice that any distance-regular graph is amply regular, and any strongly regular graph is an amply regular graph with diameter at most $2$.

In differential geometry, a general principle is that the information about the curvature at every point leads to diameter and eigenvalue bounds. Analogous results have been established for graphs. In this paper, we elaborate the point that the local regularity conditions of an amply regular graph play a role very similar to curvature. We derive sharp diameter and eigenvalue bounds for amply regular graphs by studying the so-called Lin--Lu--Yau curvature of each edge.

Ollivier \cite{Ollivier09} developed a notion of Ricci curvature of Markov chains valid on metric spaces including graphs. On graphs, the $p$-Ollivier--Ricci curvature $\kappa_p$ is defined on each edge via the Wasserstein distance between two probability measures depending on the idleness parameter $p\in [0,1]$ around the two vertices of the edge. (See Definition \ref{definition:$p$-Ollivier}). Various choices of the idleness $p$ have been studied \cite{Button3,BJL12,BM15,JL14}.
In \cite{Button4}, Lin, Lu, and Yau modified the definition of Ollivier--Ricci curvature to consider the minus of the derivative of $\kappa_p$ as a function of $p\in [0,1]$ at $p=1$. We will refer to it as the \emph{Lin--Lu--Yau curvature} and denote it by $\kappa$. The following estimates are analogues of the Bonnet--Myers theorem and Lichnerowicz estimate in Riemannian geometry.
\begin{theorem}[Discrete Bonnet--Myers theorem and Lichnerowicz estimate \cite{Ollivier09,Button4}]\label{theorem:discreteBM} Let $G=(V,E)$ be a locally finite connected graph. If $\kappa(x,y)\geq k>0$ holds true for any edge $xy\in E$, then the graph is finite and the diameter satisfies
\[\mathrm{diam}(G)\leq \frac{2}{k}.\]
The first nonzero eigenvalue $\lambda_1$ of the normalized Laplacian $L$ satisfies
\[\lambda_1\geq k.\]
Here the normalized Laplacian $L$ is defined as $L=I-D^{-\frac{1}{2}}AD^{-\frac{1}{2}}$, with $A$ being the adjacency matrix of the
graph $G$ and $D$ being the diagonal matrix of vertex degrees.
\end{theorem}
The above diameter bound is a simple consequence of the triangle inequality for the Wasserstein distance. The proof of the eigenvalue estimate involves a certain coupling method.

The curvature notions of Ollivier and Lin--Lu--Yau are naturally related to sizes of maximum matchings in appropriate subgraphs \cite{BM15,Smith14,Button8,Button7}. We develop a new method about local matching and curvature in this paper.

For any edge $xy\in E$ of an amply regular graph with parameters $(n,d,\alpha,\beta)$, it is easy to see \cite[Proposition 2.7]{Button6} that
\[\kappa(x,y)\leq \frac{2+\alpha}{d}.\]
If the amply regular graph has girth $4$, i.e., $\alpha=0, \beta\geq 2$, then it holds for any edge $xy\in E$ that (see \cite[Theorem 1.2]{Button1})
\begin{equation}\label{eq:girth4}
  \kappa(x,y)=\frac{2}{d}.
\end{equation}

For the case of girth 3, i.e. $\alpha \ge 1$, there is no exact formula for the Lin--Lu--Yau curvature purely in terms of the parameters $(n,d,\alpha,\beta)$: Bonini et al. \cite{Button7}  observed that the $4\times 4$ Rook's graph (i.e., the Cartesian product of two copies of ${K_{4}}$) and the Shrikhande graph are both strongly regular with parameters $(16,6,2,2)$, but have different curvature values $\kappa=\frac{2}{3}$ and $\kappa=\frac{1}{3}$ respectively. Li and Liu \cite{Button1} proved the following results for any edge $xy\in E$ of an amply regular graph $G=(V,E)$ with parameters $(n,d,\alpha,\beta)$ and girth 3:
\begin{itemize}
  \item [(i)] If $\alpha=1$ and $\alpha<\beta$, then
  $
  \kappa(x,y)=\frac{3}{d}.
  $
  \item [(ii)] If $\alpha\geq1$ and $\alpha=\beta-1$, then $
  \kappa(x,y)\ge \frac{2}{d}.
  $
  \item [(iii)] If $\alpha=\beta>1$, then $
  \kappa(x,y)\geq\frac{2}{d}.
  $
\end{itemize}

Our main result in this paper is the following Lin--Lu--Yau curvature estimate for amply regular graphs with girth 3.

\begin{theorem}\label{theorem:1.1}
Let $G=(V,E)$ be an amply regular graph with parameters $(n,d,\alpha,\beta)$ such that $\beta>\alpha\geq 1$.
Then the Lin--Lu--Yau curvature $\kappa(x,y)$ of any edge $xy\in E$ satisfies
\[
\kappa(x,y)\geq\frac{3}{d}.
\]
\end{theorem}
This improves and generalizes (ii) of the above-mentioned results from \cite{Button1}.
The proof is built upon a novel construction of a regular bipartite graph from the local structure of an edge $xy\in E$ (see Definition \ref{definition:3.1}). We use perfect matchings of this newly constructed graph to estimate the curvature $\kappa(x,y)$. A key step is to estimate the so-called \emph{$M$-distance} (see Definition \ref{definition:3.2}) in terms of the original graph distance (Lemma \ref{lemma:3.3}).
\begin{corollary}\label{cor:diameter}
Let $G=(V,E)$ be an amply regular graph with parameters $(n,d,\alpha,\beta)$.
\begin{itemize}
  \item [(i)] If $1\neq \beta\geq \alpha$, then \[\mathrm{diam}(G)\leq d.\]
  \item [(ii)]If $\beta>\alpha\geq 1$, then  \[\mathrm{diam}(G)\leq \left\lfloor\frac{2}{3}d\right\rfloor.\]
\end{itemize}
\end{corollary}
\begin{proof}
The bound (i) follows directly from Theorem \ref{theorem:discreteBM}, Theorem \ref{theorem:1.1}, (\ref{eq:girth4}) and the third result of \cite{Button1} mentioned above, while the bound (ii) follows from Theorem \ref{theorem:discreteBM} and Theorem \ref{theorem:1.1}.
\end{proof}
\begin{remark}
We compare our diameter bounds with known results.
\begin{itemize}
  \item [(i)] Corollary \ref{cor:diameter} (i) is already known in \cite[Theorem 1.13.2]{BCN89}. We provide here a new proof via discrete Ricci curvature.
  \item [(ii)] Neumaier and Penji\'{c} \cite[Theorem 1.1]{NP22} proved recently that for a distance-regular graph  with $d\geq 3$ and $1\neq \beta>\alpha$ we have
  \begin{equation}\label{eq:NP}
  \mathrm{diam}(G)\leq d-\beta+2\leq d.
   \end{equation}
   Recall that the $9$-Paley graph is strongly regular with parameters $(9,4,1,2)$. While (\ref{eq:NP}) produces the bound $\mathrm{diam}(G)\leq 4$, Corollary \ref{cor:diameter} (ii) produces the \emph{sharp} result $\mathrm{diam}(G)\leq 2$.
  \item [(iii)] In \cite[Corollary 1.9.2 (18) and (18a)]{BCN89}, the following two diameter bounds are proved for an amply regular graph $G$ with parameters $(n,d,\alpha,\beta)$: If $1\neq \beta\geq \alpha$ and $\mathrm{diam}(G)\geq 4$, then
      \begin{equation}\label{eq:BCN18}
      \mathrm{diam}(G)\leq d-2\beta+4\leq d.
      \end{equation}
      If $\beta\geq \max\{3,\alpha\}$ and $\mathrm{diam}(G)\geq 6$, then
      \begin{equation}\label{eq:BCN18a}
      d\geq \left(3-\frac{2}{\beta}\right)\left(\beta-3+\left\lfloor\frac{\mathrm{diam}(G)}{2}\right\rfloor\right).
      \end{equation}
      For any amply regular graph with parameters $(n,d,\alpha,\beta)$ such that $1\leq \alpha<\beta<\frac{d}{6}+2$, Corollary \ref{cor:diameter} (ii) produces better bounds than (\ref{eq:BCN18}). If $1\leq \alpha<\beta$ and $d>\frac{3}{2}(3\beta-2)(\beta-3)$, Corollary \ref{cor:diameter} (ii) produces better bounds than (\ref{eq:BCN18a}).
      Recall that the Hamming graph $H(p,q), p\geq 2, q\geq 2$ (i.e., the Cartesian product of $p$ copies of the complete graph $K_q$) is distance regular and hence amply regular with parameters $n=q^p, d=(q-1)p, \alpha=q-2, \beta=2$ \cite[Theorem 9.2.1]{BCN89}. For the Hamming graph $H(p,3)$, the bound (\ref{eq:BCN18}) yields $\mathrm{diam}(G)\leq 2p$ when $p\geq 4$, the bound (\ref{eq:BCN18a}) does not work since $\beta=2<3$, while our Corollary \ref{cor:diameter} (ii) produces $\mathrm{diam}(G)\leq \left\lfloor\frac{4p}{3}\right\rfloor$. Since the diameter of $H(p,3)$ equals $p$, we see our estimate is sharp when $p=2$. (The graph $H(2,3)$ is the $9$-Paley graph.)
\end{itemize}
\end{remark}
\begin{corollary}\label{cor:eigenvalue}
Let $G=(V,E)$ be an amply regular graph with parameters $(n,d,\alpha,\beta)$. Let $\sigma_1\leq \sigma_2\leq \cdots\leq \sigma_{n-1}\leq \sigma_n=d$ be the eigenvalues of its adjacency matrix.
\begin{itemize}
  \item [(i)] If $1\neq \beta\geq \alpha$, then \[\sigma_{n-1}\leq d-2.\]
  \item [(ii)]If $\beta>\alpha\geq 1$, then  \[\sigma_{n-1}\leq d-3.\]
\end{itemize}
\end{corollary}
\begin{proof}
This follows directly from Theorem \ref{theorem:1.1}, (\ref{eq:girth4}), the third result of \cite{Button1} mentioned above and the Lichnerowich type eigenvalue estimate $\lambda_1=1-\frac{\sigma_{n-1}}{d}\geq \inf_{xy\in E}\kappa(x,y)$ from Theorem \ref{theorem:discreteBM}.
\end{proof}
\begin{remark}\label{rmk:Koolen}
Recall that the second largest eigenvalue $\sigma_{n-1}$ of the Hamming graph $H(p,q)$ equals $(q-1)p-q=d-q$. Therefore, our estimates in Corollary \ref{cor:eigenvalue} are sharp. Indeed, the equality in Corollary \ref{cor:eigenvalue} (i) holds for $H(p,2)$ and the equality in Corollary \ref{cor:eigenvalue} (ii) holds for $H(p,3)$.
\end{remark}

To conclude, we remark another consequence of Theorem \ref{theorem:1.1} about the curvature of conference graphs.
\begin{remark}
It is conjectured by Bonini et al. \cite{Button7} that the Lin--Lu--Yau curvature of each edge in a strongly regular conference graph with parameters $(4\gamma+1, 2\gamma, \gamma-1,\gamma)$ with $\gamma\geq 2$ equals $\frac{1}{2}+\frac{1}{2\gamma}$. Theorem \ref{theorem:1.1} yields that
\[\kappa(x,y)\geq \frac{3}{2\gamma},\]
for any edge $xy$ in the conference graph. This improves the estimate $\kappa(x,y)\geq \frac{1}{\gamma}$ in \cite{Button1}.
\end{remark}

Observe that the conjecture of Bonini et al. actually claims the Lin--Lu--Yau curvature of a conference graph with parameters $(n,d,\alpha,\beta)$ achieves the upper bound $\frac{2+\alpha}{d}$. In Proposition \ref{proposition:3.1}, we give a sufficient condition for an amply regular graph with parameters $(n,d,\alpha,\beta)$ to have curvature $\frac{2+\alpha}{d}$, that is, $2\beta-\alpha\geq d+1$.

\section{Preliminaries}

We first recall the concept of Wasserstein distance between probability measures which is needed for definitions of the Ollivier--Ricci curvature and Lin--Lu--Yau curvature.

\begin{definition} [\textbf{Wasserstein distance}]\label{definition:Wasserstein}
Let $G=(V,E)$ be a locally finite graph, $\mu_{1}$ and $\mu_{2}$ be two probability measures on $G$. The Wasserstein distance   between $\mu_{1}$ and $\mu_{2}$ is defined as
\[
W(\mu_{1},\mu_{2})=\mathop{\inf}_{\pi}{\mathop\sum_{y\in V}\mathop\sum_{x\in V}d(x,y)\pi(x,y)},
\]
where $d(x,y)$ is the combinatorial distance between $x$ and $y$, i.e., the length of the shortest path connecting $x$ and $y$, and the infimum is taken over all maps $\pi:V\times V\longrightarrow[0,1]$ satisfying
\[
\mu_1(x)={\mathop\sum_{z\in V}\pi(x,z)}~~ \mu_2(y)={\mathop\sum_{z\in V}\pi(z,y)}, ~~\text{for all}~~ x,y\in V.
\]
We call such a map $\pi$ a transport plan.
\end{definition}

We consider the following particular probability measure $\mu_x^p$ with an idleness parameter $p\in[0,1]$, around a vertex $x\in V$ :
\[
\mu_x^p(v)=\left\{
                    \begin{array}{ll}
                      p, & \hbox{if $v=x$;} \\
                      \frac{1-p}{\mathrm{deg}(x)}, & \hbox{if $v\sim x$;} \\
                      0, & \hbox{otherwise,}
                    \end{array}
                  \right.
\]
where $\mathrm{deg}(x):={\mathop\sum\limits_{v\in V:v\sim x}{1}}$ denotes the vertex degree of $x$.

\begin{definition} [$p$-Ollivier--Ricci curvature \cite{Ollivier09,Button3} and Lin--Lu--Yau curvature \cite{Button4}]\label{definition:$p$-Ollivier}
Let $G=(V,E)$ be a locally finite graph. For any vertices $x,y\in V$, the $p$-Ollivier--Ricci curvature $\kappa_p(x,y)$, $p\in[0,1]$, is defined as
\[
\kappa_p(x,y)=1-\frac{W(\mu_x^p,\mu_y^p)}{d(x,y)}.
\]
The Lin--Lu--Yau curvature $\kappa(x,y)$ is defined as
\[
\kappa(x,y)=\mathop{\lim}_{p\rightarrow1}{\frac{\kappa_p(x,y)}{1-p}}.
\]
\end{definition}

Given two vertices $x$ and $y$, Lin--Lu--Yau curvature equals to the minus of the derivative of the function $p\mapsto \kappa_p(x,y)$ at $p=1$, since $\kappa_1(x,y)=0$.

For an edge $xy$ in a $d$-regular graph, Bourne et al. \cite{Button5} showed that the function $p\mapsto \kappa_p(x,y)$ is concave, piecewise linear over $[0,1]$, and, in particular, linear over $[\frac{1}{d+1},1]$. This leads to a limit-free reformulation of the Lin--Lu--Yau curvature along an edge:
\begin{equation}\label{eq:Bourne}
\kappa(x,y)=\frac{d+1}{d}\kappa_\frac{1}{d+1}(x,y).
\end{equation}

We recall the important concept of matching from graph theory.
\begin{definition}\cite[Section 16.1]{BM08}
Let $G=(V,E)$ be a locally finite simple connected graph. A set $M$ of pairwise nonadjacent edges is called a \emph{matching}. The two vertices of each edge of $M$ are said to be \emph{matched} under $M$, and each vertex adjacent to an edge of $M$ is said to be \emph{covered} by $M$. A matching $M$ is called a \emph{perfect matching} if it covers every vertex of the graph.
\end{definition}

The following Hall's marriage theorem \cite{Hall35} (see also \cite[Theorem 16.4]{BM08}) deals with the existence of a perfect matching in a bipartite graph.
\begin{theorem}[Hall's Marriage Theorem]\label{theorem:Hall's}
Let $H=(V,E)$ be a bipartite graph with the bipartition $V=X\sqcup Y$. Then $H$ has a perfect matching if and only if
\[
|X|=|Y|\,\, \text{and}\,\,  |\Gamma_H(S)|\geq|S|\,\, \text{for all} \, \,S\subseteq X\,(or Y)
\]
holds, where $\Gamma_H(S):=\{v\in V|\,\, \text{there exists}\,\, w\in S \,\, \text{such that}\,\, vw\in E\}$.
\end{theorem}

The following K\"onig's theorem \cite{Konig16} (see also \cite[Theorem 14 in Chap. XI]{Konig36}) is an important tool for our purpose. It is a direct consequence of Hall's marriage theorem.
\begin{theorem}[K\"onig's theorem]\label{lemma:Konig} A bipartite graph $G$ can be decomposed into $d$ edge-disjoint perfect matchings if and only if $G$ is $d$-regular.
\end{theorem}

Another direct consequence of the Hall's marriage theorem, which is useful for our later purposes, is given below.
\begin{corollary}\label{corollary:exist optimal matching 2}
Let $H=(V,E)$ be a bipartite graph with the bipartition $V=X\sqcup Y$, and $|X|=|Y|=n$ $(n\geq1)$. If the minimal vertex degree $\delta(H)\geq\frac{n}{2}$, then $H$ has a perfect matching.
\end{corollary}

\section{Proofs of the main results}
Let $G=(V,E)$ be a graph. For any vertex $v\in V$, we denote by $\Gamma(v):=\{w\in V | vw\in E\}$ the set of its neighbors. In what follows, we fix an edge $x y\in E$ and use the notations:
\[\Delta_{xy}:=\Gamma(x)\cap\Gamma(y),\,\,N_{x}:=\Gamma(x)\setminus(\{y\}\cup\Gamma(y))\,\,\text{and}\,\,N_{y}:=\Gamma(y)\setminus(\{x\}\cup\Gamma(x)).\]

Before presenting the proof of Theorem \ref{theorem:1.1}, we prepare several definitions and lemmas.
We first construct a bipartite graph from the local structure of an edge $xy\in E$ as follows.
\begin{definition} \label{definition:3.1}
Let $G=(V,E)$ be an amply regular graph with parameters $(n,d,\alpha,\beta)$ such that $\beta>\alpha\geq 1$.
For any edge $xy\in E$, we construct a bipartite graph $H=(V_H, E_H)$ as follows. Let us denote the $\alpha$ vertices in $\Delta_{xy}$ by $z_1,\cdots,z_\alpha$. The vertex set of $H$ is given by
\[V_H=(N_x\cup \Delta_{xy}\cup \{x_1,\cdots,x_{\beta-\alpha-1}\})\sqcup (N_y\cup \Delta_{xy}'\cup \{x_1',\cdots,x_{\beta-\alpha-1}'\}).\]
Here $\Delta'_{xy}:=\{z_1',\ldots, z'_\alpha\}$ is a new added set with $\alpha$ vertices, which is considered as a copy of $\Delta_{xy}$. The sets $\{x_1,\cdots,x_{\beta-\alpha-1}\}$ and $\{x_1',\cdots,x_{\beta-\alpha-1}'\}$ are new added set where $\beta-\alpha\geq 1$, which are considered as copies of the vertex $x$ and are empty if $\beta-\alpha=1$.

The edge set $E_H:=\bigcup_{i=1}^8 E_i$ is given by
\begin{align*}
E_1=&\{vw | v\in N_x, w\in N_y, vw\in E\},\\
E_2=&\{vz_i' | v\in N_x, z_i'\in \Delta_{xy}', vz_i\in E\},\\
E_3=&\{z_iw | z_i\in \Delta_{xy}, w\in N_y, z_iw\in E\},\\
E_4=&\{z_iz_i' | i=1,\cdots,\alpha\},\\
E_5=&\{z_iz_j' | z_iz_j\in E, i\neq j, 1\leq i,j\leq \alpha\},\\
E_6=&\{x_iz_j' | i=1,\cdots,\beta-\alpha-1;j=1,\cdots,\alpha\},\\
E_7=&\{x_i'z_j | i=1,\cdots,\beta-\alpha-1;j=1,\cdots,\alpha\},\\
E_8=&\{x_ix_j' | 1\leq i,j\leq \beta-\alpha-1\}.
\end{align*}
In the above $E_6, E_7$ and $E_8$ are empty when $\beta-\alpha=1$.
We call $H$ a \emph{transport-bipartite graph} of $xy$.
\end{definition}

\begin{figure}[H]
    \centering
    \includegraphics[width=7cm]{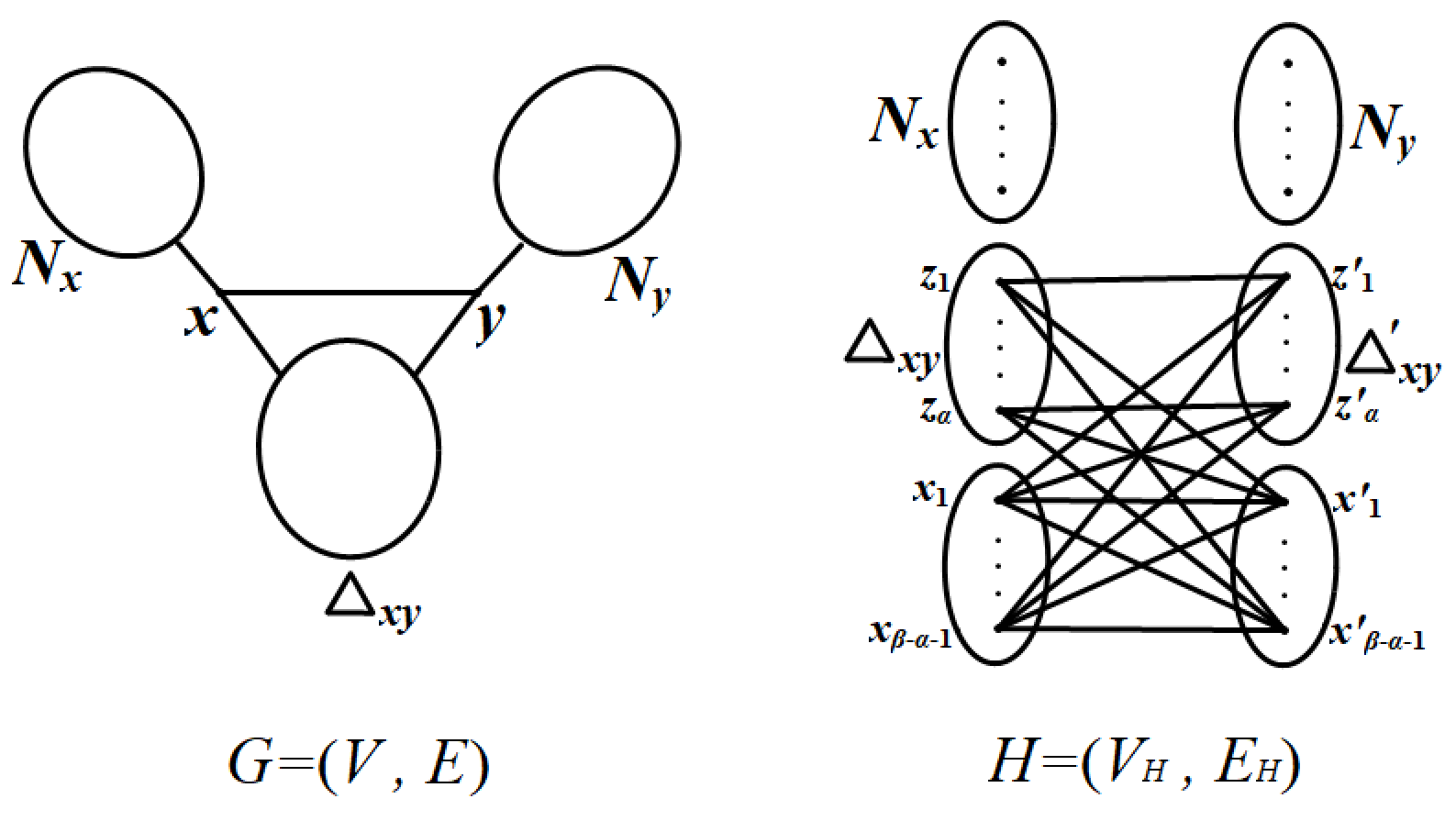}
    \caption{A schematic plot for Definition \ref{definition:3.1}}\label{fig:0}
\end{figure}

The definition of a transport-bipartite graph is depicted in Figure \ref{fig:0}. Note that only the edge sets $E_4, E_6, E_7$ and $E_8$ are displayed in the graph $H$. The other edge sets depend on the particular structure of graph $G$.

\begin{lemma}\label{lemma:3.1}
Let $G=(V,E)$ be an amply regular graph with parameters $(n,d,\alpha,\beta)$ where $\beta>\alpha\geq 1$. Then the transport-bipartite graph $H$ of any edge $xy\in E$ is $(\beta-1)$-regular.
\end{lemma}

\begin{proof}
For any $xy\in E$, we check the degree of vertices from $V_H$ in $H=(V_H, E_H)$ as follows.
\begin{itemize}
  \item [(a)] For any vertex $v\in N_x$, we have $d(v,y)=2$ in $G$. Therefore, there are $\beta$ common neighbors of $v$ and $y$ in $V$, $\beta-1$ vertices of which lie in $N_y\cup \Delta_{xy}$. Therefore, in the graph $H$, there are $\beta-1$ neighbors of $v$ in $N_y\cup\Delta_{xy}'$ by definition.
 Similarly, for any vertex $v\in N_y$, there are $\beta-1$ neighbors of $v$ in $N_x\cup\Delta_{xy}$.
  \item [(b)] For any vertex $v\in\Delta_{xy}$, we have $vy\in E$. Therefore, there are $\alpha$ common neighbors of $v$ and $y$ in $V$, $\alpha-1$ vertices of which lie in $N_y\cup \Delta_{xy}$. Therefore, in the graph $H$, there are $\alpha$ neighbors of $v$ in $N_y\cup\Delta_{xy}'$ by definition (Recall the existence of $E_4$). Due to the edge set $E_6$, there are $\alpha+(\beta-\alpha-1)=\beta-1$ neighbors of $v$ in $N_y\cup\Delta_{xy}'\cup\{x_1',\cdots,x_{\beta-\alpha-1}'\}$. Similarly, for any vertex $v\in \Delta_{xy}'$, there are $\beta-1$ neighbors of $v$ in $N_x\cup\Delta_{xy}\cup\{x_1,\cdots,x_{\beta-\alpha-1}\}$.
  \item [(c)] When $\beta-\alpha>1$, for any vertex $v\in \{x_1,\cdots,x_{\beta-\alpha-1}\}$, there are $\alpha$ neighbors of $v$ in $\Delta_{xy}'$ and $\beta-\alpha-1$ neighbors in $\{x_1',\ldots,x'_{\beta-\alpha-1}\}$. Similarly, for any vertex $v\in \{x_1',\cdots,x_{\beta-\alpha-1}'\}$, there are $\alpha+(\beta-\alpha-1)=\beta-1$ neighbors of $v$ in $\Delta_{xy}\cup\{x_1,\cdots,x_{\beta-\alpha-1}\}$.
\end{itemize}
In conclusion, every vertex $v\in V_H$ has $\beta-1$ neighbors in $H$.
\end{proof}

\begin{definition} \label{definition:3.2}
Let $G=(V,E)$ be an amply regular graph with parameters $(n,d,\alpha,\beta)$ where $\beta>\alpha\geq 1$. For any edge $xy\in E$, let $H$ be the transport-bipartite graph of $xy$ and $M$ be a perfect matching of $H$. Note that the existence of such a perfect matching $M$ is guaranteed by Theorem \ref{lemma:Konig} and Lemma \ref{lemma:3.1}. We say $v_0\in N_x$ and $w_0\in N_y$ are \emph{reachable in $M$} if either $v_0w_0\in M$ or there exist $\{t_1,\cdots,t_j\}\subset\Delta_{xy}\cup\{x_1,\cdots,x_{\beta-\alpha-1}\}$ for some $1\leq j\leq\beta-1$ such that
\[\{t_0t_1',t_1t_2',\cdots,t_{j-1}t_j',t_jw_0\}\subset M\]
where we use the notation $t_0:=v_0$. We denote it as $v_0\stackrel{M}{\longleftrightarrow}w_0$.
We call \[C_{v_0\stackrel{M}{\longleftrightarrow}w_0}:=(v_0,t_1,\cdots,t_j,w_0)\] the \emph{$M$-chain} between $v_0$ and $w_0$ in $M$. The length of the M-chain between $v_0$ and $w_0$ is defined as the \emph{$M$-distance} between $v_0$ and $w_0$ in $M$, denoted by $\rho_M(v_0,w_0)$.
\end{definition}

Note that the $M$-chain between two reachable vertices is unique, since every vertex in a $M$-chain is only covered by a unique edge in the perfect matching $M$. Therefore, the $M$-distance $\rho_M(v_0,w_0)$ is well-defined.
We have $\rho_M(v_0,w_0)\geq d(v_0,w_0)$, where $d(\cdot,\cdot)$ is the combinatorial distance in the graph $G$.
Observe that \[\rho_M(v_0,w_0)=|C_{v_0\stackrel{M}{\longleftrightarrow}w_0}\cap(N_x\cup\Delta_{xy}\cup\{x_1,\cdots,x_{\beta-\alpha-1}\})|.\]

\begin{lemma}\label{lemma:3.2}
Let $G=(V,E)$ be an amply regular graph with parameters $(n,d,\alpha,\beta)$ where $\beta>\alpha\geq 1$. For any edge $xy\in E$, let $H$ be the transport-bipartite graph of $xy$ and $M$ be a perfect matching of $H$. Let $\phi: N_x\to N_y$ be a map such that $v\stackrel{M}{\longleftrightarrow}\phi(v)$. Then the map $\phi$ is well-defined and one-to-one. Moreover, the corresponding $M$-chains are pairwise disjoint.
\end{lemma}

\begin{proof}
First, we show that for every $v\in N_x$, there exists a $w\in N_y$ such that $v\stackrel{M}{\longleftrightarrow}w$. If there exists a $w\in N_y$ such that $vw\in M$, we have $v\stackrel{M}{\longleftrightarrow}w$ by definition. If, otherwise, $vw\notin M$ for any $w\in N_y$,
there must be a $t_1\in \Delta_{xy}$ such that $vt_1'\in M$.
\\
\textbf{Step 1.} Let $t_1u_1\in M$ be the unique edge in $M$ covering $t_1$. Notice that
\[u_1\in N_y\cup\Delta_{xy}'\cup\{x_1',\cdots,x_{\beta-\alpha-1}'\}.\]
If $u_1\in N_y$, then we have $v\stackrel{M}{\longleftrightarrow}u_1$ by definition. Otherwise, we have $u_1\in \Delta_{xy}'\cup\{x_1',\cdots,x_{\beta-\alpha-1}'\}$ and we set $t_2':=u_1$.
\\
\textbf{Step 2. }Let $t_2u_2\in M$ be the unique edge in $M$ that covers $t_2$. Then, we have
\[u_2\in N_y\cup\Delta_{xy}'\cup\{x_1',\cdots,x_{\beta-\alpha-1}'\}.\]
If $u_2\in N_y$, then we have $v\stackrel{M}{\longleftrightarrow}u_2$ by definition. Otherwise, we have $u_2\in \Delta_{xy}'\cup\{x_1',\cdots,x_{\beta-\alpha-1}'\}$ and we set $t_3':=u_2$.

We continue this process to obtain a sequence of vertices $t_2'=u_1, t_3'=u_2,\ldots$ from the set $N_y\cup\Delta_{xy}'\cup\{x_1',\cdots,x_{\beta-\alpha-1}'\}$ and stop when $u_j\in N_y$ for the first time. Since the set $\Delta_{xy}'\cup\{x_1',\cdots,x_{\beta-\alpha-1}'\}$ is finite, and $u_j$'s are distinct, this process must stop after a finite step of iterations. That is, there exists a finite $j_0$ such that $u_{j_0}\in N_y$ for the first time. Then we have $v\stackrel{M}{\longleftrightarrow}u_{j_0}$ by definition.

Next, we show the map $\phi$ is one-to-one. For any two distinct vertices $v_1$ and $v_2$ in $N_x$, if $\phi(v_1)=\phi(v_2)$, then there exits two edges in $M$ covering the same vertex, which is a contradiction. Therefore, the map $\phi$ is injective. Since $|N_x|=|N_y|$, the map $\phi$ must be a bijection. By a similar argument, we further derive that the corresponding $M$-chains are pairwise disjoint.
\end{proof}

\begin{lemma}\label{lemma:3.3}
Let $G=(V,E)$ be an amply regular graph with parameters $(n,d,\alpha,\beta)$ where $\beta>\alpha\geq 1$. For any edge $xy\in E$, let $H$ be the transport-bipartite graph of $xy$  and $M$ be a perfect matching of $H$. For any $v_0\in N_x, w_0\in N_y$ which are reachable in $M$, we have
\[
d(v_0,w_0)\leq\rho_M(v_0,w_0)-k,
\]
where  $k=|C_{v_0\stackrel{M}{\longleftrightarrow}w_0}\cap\{x_1,\cdots,x_{\beta-\alpha-1}\}|$.
\end{lemma}

\begin{proof}
Let $\delta:=\beta-\alpha$. If $\delta=1$, then $k=0$, the estimate holds true.
Next, we suppose $\delta>1$. Notice that $0\leq k\leq \delta-1$. If $k=0$, the estimate holds true. It remains to discuss the case $1\leq k\leq \delta-1$.
Recall that
\[\rho_M(v_0,w_0)=|C_{v_0\stackrel{M}{\longleftrightarrow}w_0}\cap(N_x\cup\Delta_{xy}\cup\{x_1,\cdots,x_{\delta-1}\})|.\]
Notice that we have $|C_{v_0\stackrel{M}{\longleftrightarrow}w_0}\cap N_x|=|\{v_0\}|=1$. Now we estimate $|C_{v_0\stackrel{M}{\longleftrightarrow}w_0}\cap\Delta_{xy}|$.

Let $v_0u_1\in M$ be the unique edge in $M$ which covers $v_0$. Here we have $u_1\in N_y\cup\Delta_{xy}'$. If $u_1\in N_y$, then we have $u_1=w_0$ by Lemma \ref{lemma:3.2}. Moreover, we have \[|C_{v_0\stackrel{M}{\longleftrightarrow}w_0}\cap\{x_1,\cdots,x_{\delta-1}\}|=|\{v_0,w_0\}\cap\{x_1,\cdots,x_{\delta-1}\}|=0,\]
which contradicts the assumption that $k\geq 1$. So we have $u_1\in \Delta_{xy}'$. Let us set $t_1':=u_1$ with $t_1\in \Delta_{xy}$.
Let $u^*w_0\in M$ be the unique edge in $M$ which covers $w_0$, where $u^*\in N_x\bigcup\Delta_{xy}$. Similarly, we have $u^*\in \Delta_{xy}$.

We claim that $u^*\neq t_1$.
If $u^*=t_1$, then we have, by Lemma \ref{lemma:3.2}, that $C_{v_0\stackrel{M}{\longleftrightarrow}w_0}=(v_0,t_1,w_0)$. Hence $|C_{v_0\stackrel{M}{\longleftrightarrow}w_0}\cap\{x_1,\cdots,x_{\delta-1}\}|=0$, which contradicts the assumption that $k\geq 1$.
In conclusion, we have
      \[
      |C_{v_0\stackrel{M}{\longleftrightarrow}w_0}\cap\Delta_{xy}|\geq2.
      \]
      Then we have
      \begin{equation}
        \begin{split}
           \rho_M(v_0,w_0)=&|C_{v_0\stackrel{M}{\longleftrightarrow}w_0}\cap(N_x\cup\Delta_{xy}\cup\{x_1,\cdots,x_{\delta-1}\})|\\
                          =&|C_{v_0\stackrel{M}{\longleftrightarrow}w_0}\cap N_x|+|C_{v_0\stackrel{M}{\longleftrightarrow}w_0}\cap \Delta_{xy}|
                          +|C_{v_0\stackrel{M}{\longleftrightarrow}w_0}\cap\{x_1,\cdots,x_{\delta-1}\}|\\
                          \geq & 1+2+k\\
                          =&k+3.
        \end{split}
        \nonumber
      \end{equation}
      Since $d(v_0,w_0)\leq3$, we derive that
      $
      d(v_0,w_0)\leq\rho_M(v_0,w_0)-k.
      $
\end{proof}

Now, we are prepared for the proof of Theorem \ref{theorem:1.1}.
\begin{proof}[Proof of Theorem \ref{theorem:1.1}]
For any edge $xy\in E$, let $H$ be the transport-bipartite graph of $xy$. By Theorem \ref{lemma:Konig}, there exists a perfect matching $M$ of $H$ which covers $z_1z_1'\in E_H$.
We consider the following particular transport plan $\pi_0:V\times V\rightarrow[0,1]$ from $\mu_x^{\frac{1}{d+1}}$ to $\mu_y^{\frac{1}{d+1}}$:
\[
\pi_0(v,w):=\left\{
                    \begin{array}{ll}
                      \frac{1}{d+1}, &\hbox{if $v=w\in \Delta_{xy}\bigcup\{x,y\}$;} \\
                      \frac{1}{d+1}, &\hbox{if $v\in N_x,w\in N_y,\,\text{and}\,v\stackrel{M}{\longleftrightarrow}w$;} \\
                      0, &\hbox{otherwise.}
                    \end{array}
                  \right.
\]
Notice that $\pi_0$ is well defined by Lemma \ref{lemma:3.2}.
We calculate
\begin{equation}
  \begin{split}
    W(\mu_x^{\frac{1}{d+1}},\mu_y^{\frac{1}{d+1}})&=\mathop{\inf}_{\pi}{\mathop\sum_{v\in V}\mathop\sum_{w\in V}d(v,w)\pi(v,w)}\\
                                                  &\leq\mathop\sum_{v\in V}\mathop\sum_{w\in V}d(v,w)\pi_0(v,w)\\
                                                  &=\frac{1}{d+1}\mathop\sum_{v\in N_x,w\in N_y,\,\text{with}\,v\stackrel{M}{\longleftrightarrow}w}d(v,w).
  \end{split}
  \nonumber
\end{equation}

Let us denote $N_x:=\{v_1,\ldots, v_{p}\}$ and $N_y:=\{w_1,\ldots,w_{p}\}$ with $p=d-\alpha-1$ such that $v_i$ and $w_i$ are reachable in $M$ for any $i=1,\ldots,p$.
Set $\delta:=\beta-\alpha$ and $k_i:=|C_{v_i\stackrel{M}{\longleftrightarrow}w_i}\cap\{x_1,\cdots,x_{\delta-1}\}|$.
By Lemma \ref{lemma:3.3}, we derive
\begin{align}\label{eq:W}
(d+1)W(\mu_x^{\frac{1}{d+1}},\mu_y^{\frac{1}{d+1}})\leq\sum_{i=1}^pd(v_i,w_i)\leq \sum_{i=1}^p[\rho_M(v_i,w_i)-k_i]=\sum_{i=1}^{p}\rho_M(v_i,w_i)-k,
\end{align}
where $k:=\sum_{i=1}^pk_i$.

We notice that $k=|(\bigcup_{i = 1}^{p}C_{v_i\stackrel{M}{\longleftrightarrow}w_i})\cap\{x_1,\cdots,x_{\delta-1}\}|$. Hence,
\[\left|\{x_1,\cdots,x_{\delta-1}\}\setminus\bigcup_{i = 1}^{p}C_{v_i\stackrel{M}{\longleftrightarrow}w_i}\right|=\delta-1-k.\]
Since $z_1z_1'\in M$, we have $z_1\not\in \bigcup_{i = 1}^{p}C_{v_i\stackrel{M}{\longleftrightarrow}w_i}$ by definition. Therefore,
we have
\[\left|\left(\{z_1\}\bigcup\{x_1,\cdots,x_{\delta-1}\}\right)\setminus\bigcup\limits_{i = 1}^{p}C_{v_i\stackrel{M}{\longleftrightarrow}w_i}\right|=\delta-k.\]
We estimate that
\begin{equation}
  \begin{split}
\sum_{i = 1}^{p}\rho_M(v_i,w_i)
       &=\left|\left(N_x\cup\Delta_{xy}\cup\{x_1,\cdots,x_{\delta-1}\}\right)\cap (\bigcup\limits_{i = 1}^{p}C_{v_i\stackrel{M}{\longleftrightarrow}w_i})\right|\\
       &=\left|N_x\cup\Delta_{xy}\cup\{x_1,\cdots,x_{\delta-1}\}\right|\\
       &\,\,\,\,\,\,\,\,-\left|(N_x\cup\Delta_{xy}\cup\{x_1,\cdots,x_{\delta-1}\})\setminus\bigcup\limits_{i = 1}^{p}C_{v_i\stackrel{M}{\longleftrightarrow}w_i}\right|\\
       &\leq|N_x\cup\Delta_{xy}\cup\{x_1,\cdots,x_{\delta-1}\}|\\
       &\,\,\,\,\,\,\,\,-\left|(\{z_1\}\cup\{x_1,\cdots,x_{\delta-1}\})\setminus\bigcup\limits_{i = 1}^{p}C_{v_i\stackrel{M}{\longleftrightarrow}w_i}\right|\\
       &=(d+\delta-2)-(\delta-k)\\
       &=d+k-2.
  \end{split}
  \nonumber
\end{equation}
Inserting into (\ref{eq:W}) yields
\begin{equation}
  \begin{split}
    W(\mu_x^{\frac{1}{d+1}},\mu_y^{\frac{1}{d+1}})
   \leq\frac{d-2}{d+1}.
  \end{split}
  \nonumber
\end{equation}
Then we obtain by (\ref{eq:Bourne}) that
\[
\kappa(x,y)=\frac{d+1}{d}(1-W(\mu_x^{\frac{1}{d+1}},\mu_y^{\frac{1}{d+1}}))\geq\frac{d+1}{d}\left(1-\frac{d-2}{d+1}\right)=\frac{3}{d}.
\]
This completes the proof.
\end{proof}

\begin{proposition}\label{proposition:3.1}
Let $G=(V,E)$ be an amply regular graph with parameters $(n,d,\alpha,\beta)$ such that $2\beta-\alpha\geq d+1$. Then we have for any $xy\in E$ that
\[
\kappa(x,y)=\frac{2+\alpha}{d}.
\]
\end{proposition}

\begin{proof}
For any $xy\in E$, we consider the bipartite graph $B=(N_x\sqcup N_y, E_1)$. Recall that $E_1=\{vw | v\in N_x, w\in N_y, vw\in E\}$. For any subset $S\subset V$ and vertex $v\in V$, we denote the set of neighbors of $v$ in $S$ by $\Gamma_S(\{v\}):=\{u\in S: uv\in E\}$.
For any $v\in N_x$, we have
\[|\Gamma_{N_y\cup \Delta_{xy}}(\{v\})|=\beta-1\,\,\text{ and }\,\,|\Gamma_{\Delta_{xy}}(\{v\})|\leq |\Delta_{xy}|=\alpha.\]
Then we get $|\Gamma_{N_y}(\{v\})|\geq\beta-\alpha-1$. Similarly, we have for any $v\in N_y$ that $|\Gamma_{N_x}(\{v\})|\geq\beta-\alpha-1$. That is, the minimal vertex degree $\delta(B)\geq\beta-\alpha-1$. By the assumption that $2\beta-\alpha\geq d+1$, we derive
\[
\delta(B)\geq\frac{d-\alpha-1}{2}=\frac{|N_x|}{2}.
\]
Applying Corollary \ref{corollary:exist optimal matching 2}, there exists a perfect matching of $B$. Then we check directly
\[
\kappa(x,y)=\frac{d+1}{d}(1-W(\mu_x^{\frac{1}{d+1}},\mu_y^{\frac{1}{d+1}}))=\frac{d+1}{d}\left(1-\frac{d-\alpha-1}{d+1}\right)=\frac{2+\alpha}{d}.
\]
\end{proof}
Notice that we allow $\alpha=0$ in Proposition \ref{proposition:3.1}.
In Figures \ref{fig:1} and \ref{fig:2}, we give two examples of amply regular graphs satisfying the parameter condition in Proposition \ref{proposition:3.1}.
\begin{figure}[H]
    \begin{minipage}{6cm}
    \centering
    \includegraphics[height=3cm,width=3cm]{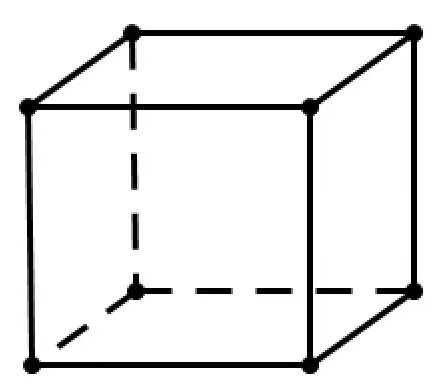}
    \caption{(8,3,0,2)}\label{fig:1}
    \end{minipage}
    \begin{minipage}{6cm}
    \centering
    \includegraphics[height=3cm,width=3cm]{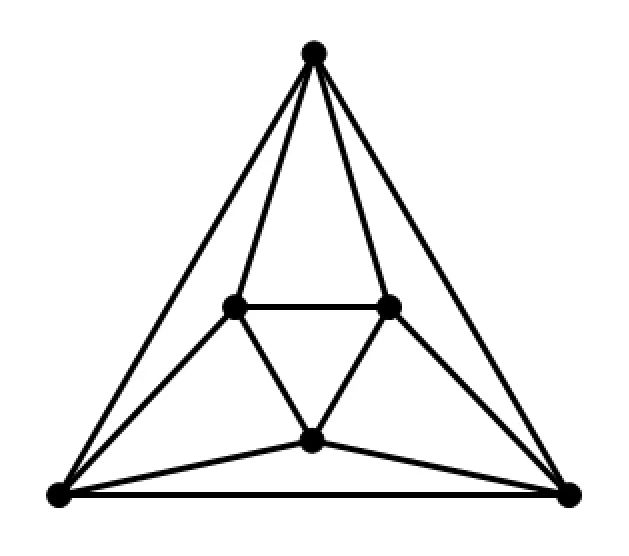}
    \caption{(6,4,2,4)}\label{fig:2}
    \end{minipage}
\end{figure}

\begin{remark}\label{remark:3.1}
Let $G=(V,E)$ be an amply regular graph with parameters $(n,d,\alpha,\beta)$ and girth 3 $(\alpha\geq1)$. By Theorem \ref{theorem:1.1}, \cite[Theorem 1.3]{Button1} and the upper bound in \cite[Proposition 2.7]{Button6}, we derive the following estimates: For any $xy\in E$, if $\beta=\alpha>1$, then
\[
\frac{2+\alpha}{d}\geq\kappa(x,y)\geq\frac{2}{d};
\]
If $\beta>\alpha\geq1$, then
\[
\frac{2+\alpha}{d}\geq\kappa(x,y)\geq\frac{3}{d}.
\]
We collect examples below showing that each of these inequalities is sharp.
\\For the case that $\beta=\alpha>1$, we have:
\begin{itemize}
  \item [(1)] The Shrikhande graph is strongly regular with parameters $(16,6,2,2)$, for which the Lin--Lu--Yau curvature of each edge is $\kappa=\frac{1}{3}=\frac{2}{d}$;
  \item [(2)] The $4\times 4$ Rook's graph is strongly regular with parameters $(16,6,2,2)$, for which the Lin--Lu--Yau curvature of each edge satisfies $\kappa=\frac{2}{3}=\frac{2+\alpha}{d}$.
\end{itemize}
For the case that $\beta>\alpha\geq1$, we have:
\begin{itemize}
  \item [(1)] For any amply regular graph with parameters $(n,d,1,\beta)$, the Lin--Lu--Yau curvature of each edge is $\kappa=\frac{3}{d}$ \cite[Theorem 1.3]{Button1};
  \item [(2)] For any amply regular graph with girth $3$ and parameters $(n,d,\alpha,\beta)$ such that $2\beta-\alpha\geq d+1$, the Lin--Lu--Yau curvature of each edge is $\kappa=\frac{2+\alpha}{d}$ according to Proposition \ref{proposition:3.1}. For example, the graph in Figure \ref{fig:2}, which is the Johnson graph $J(4,2)$, is such a case.
\end{itemize}
\end{remark}
\section*{Acknowledgements}
This work is supported by the National Key R and D Program of China 2020YFA0713100.
XH is supported by the National Natural Science Foundation of China (No. 11601238 and No. 12371206). SL is supported by
the National Natural Science Foundation of China (No. 12031017).
We warmly thank Jack H. Koolen for his kind comments and advices,  particularly for pointing out the sharpness of our eigenvalue estimates via the example of Hamming graphs in Remark \ref{rmk:Koolen}. We are grateful to the anonymous referees for their careful reading of our manuscript and their very helpful comments.

\end{document}